\documentclass{article} 
\usepackage{amsmath,amsthm}     
\usepackage{graphicx}     
\usepackage{hyperref} 
\usepackage{url}
\usepackage{amsfonts} 
\usepackage{multicol}
\usepackage[utf8]{inputenc}

\newtheorem{theorem}{Theorem}
\newtheorem{example}{Example}
\theoremstyle{definition}
\newtheorem*{definition}{Definition}
\newtheorem*{remark}{Remark}

\begin{document}

\title{Sets of iterated Partitions and the\\ Bell iterated Exponential Integers}

\author{Ivar Henning Skau\\               
University of South-Eastern Norway\\    
3800 B{\o}, Telemark\\              
ivar.skau@usn.no
\and
Kai Forsberg Kristensen\\            
University of South-Eastern Norway\\    
3918 Porsgrunn, Telemark\\          
kai.f.kristensen@usn.no}                   

\maketitle 

\begin{abstract}
It is well known that the Bell numbers represent the total number of partitions of an n-set. Similarly, the Stirling numbers of the second kind, represent the number of k-partitions of an n-set. In this paper we introduce a certain partitioning process that gives rise to a sequence of sets of "nested" partitions. We prove that at stage m, the cardinality of the resulting set will equal the m-th order Bell number. This set-theoretic interpretation enables us to make a natural definition of higher order Stirling numbers and to study the combinatorics of these entities. The cardinality of the elements of the constructed "hyper partition" sets are explored.
\end{abstract} 

\section{A partitioning process.}

Consider the 3-set $S=\{a,b,c\}$. The partition set $\wp_3^{(1)}$ of $S$, where the elements of $S$ are put into boxes, contains the five partitions shown in the second column of Figure \ref{fig:partitions}.\\

\begin{figure}[htbp]
	\centering
	\setlength{\unitlength}{1cm}
\begin{picture}(15,8)
\put(0.4,8){$S=\wp_3^{(0)}$}
\put(2.8,8){$\wp_3^{(1)}$}
\put(6.7,8){$\wp_3^{(2)}$}
\put(2.2,7.0){$\{$}
\put(2.5,6.9){\framebox(0.8,0.45){$abc$}$\:,$}
\put(2.5,6.0){\framebox(0.6,0.45){$ab$}~\framebox(0.4,0.45){$c$}$\:,$}
\put(2.5,5.1){\framebox(0.6,0.45){$ac$}~\framebox(0.4,0.45){$b$}$\:,$}
\put(2.5,4.2){\framebox(0.6,0.45){$bc$}~\framebox(0.4,0.45){$a$}$\:,$}
\put(2.5,3.3){\framebox(0.4,0.45){$a$}~\framebox(0.4,0.45){$b$}~\framebox(0.4,0.45){$c$}}
\put(4.1,3.4){$\}$}
\put(0.3,7){$\{a,b,c\}$}

\put(4.8,7.05){$\{$}
\put(5.1,6.85){\framebox(1,0.6){\framebox(0.8,0.45){$abc$}}$\:,$}
\put(5.1,5.95){\framebox(1.3,0.6){\framebox(0.6,0.45){$ab$}~\framebox(0.4,0.45){$c$}}$\:,$~\framebox(0.78,0.6){\framebox(0.6,0.45){$ab$}}~\framebox(0.58,0.6){\framebox(0.4,0.45){$c$}}$\:,$}
\put(5.1,5.05){\framebox(1.3,0.6){\framebox(0.6,0.45){$ac$}~\framebox(0.4,0.45){$b$}}$\:,$~\framebox(0.78,0.6){\framebox(0.6,0.45){$ac$}}~\framebox(0.58,0.6){\framebox(0.4,0.45){$b$}}$\:,$}
\put(5.1,4.15){\framebox(1.3,0.6){\framebox(0.6,0.45){$bc$}~\framebox(0.4,0.45){$a$}}$\:,$~\framebox(0.78,0.6){\framebox(0.6,0.45){$bc$}}~\framebox(0.58,0.6){\framebox(0.4,0.45){$a$}}$\:,$}
\put(5.1,3.25){\framebox(1.65,0.6){\framebox(0.4,0.45){$a$}~\framebox(0.4,0.45){$b$}~\framebox(0.4,0.45){$c$}}$\:,$~\framebox(1.1,0.6){\framebox(0.4,0.45){$a$}~\framebox(0.4,0.45){$b$}}~\framebox(0.58,0.6){\framebox(0.4,0.45){$c$}}$\:,$}
\put(5.1,2.35){\framebox(1.1,0.6){\framebox(0.4,0.45){$a$}~\framebox(0.4,0.45){$c$}}~\framebox(0.58,0.6){\framebox(0.4,0.45){$b$}}$\:,$~\framebox(1.1,0.6){\framebox(0.4,0.45){$b$}~\framebox(0.4,0.45){$c$}}~\framebox(0.58,0.6){\framebox(0.4,0.45){$a$}}$\:,$~\framebox(0.58,0.6){\framebox(0.4,0.45){$a$}}~\framebox(0.58,0.6){\framebox(0.4,0.45){$b$}}~\framebox(0.58,0.6){\framebox(0.4,0.45){$c$}}}
\put(11.35,2.5){$\}$}

\put(1.95,8.5){\line(0,-1){6.3}}
\put(4.5,8.5){\line(0,-1){6.3}}
\put(0,7.7){\line(1,0){12}}

\end{picture}		
\vspace{-2.4cm}
\caption{The basic set $S$ together with the partition sets $\wp_3^{(1)}$ and $\wp_3^{(2)}$}
	\label{fig:partitions}
\end{figure}  

\vspace{0.5cm}

\noindent Now we proceed, putting boxes into boxes. This means that we create a second order partition set to each first order partition in $\wp_3^{(1)}$. The union of all the second order partition sets is denoted by $\wp_3^{(2)}$, and appears in the third column of Figure \ref{fig:partitions}.\\

\begin{definition}
$\wp_n^{(1)}$ is the set of all partitions of a given $n$-set. For $m>1$, $\wp_n^{(m)}$, called the $m$-th order partition set, is the union of the complete collection of sets, each being the partition set of an element in $\wp_n^{(m-1)}$.
\end{definition}

\noindent We observe that the number of partitions in $\wp_3^{(2)}$ is $|\wp_3^{(2)}|=12$, which also, and not by coincidence, turns out to be the second order Bell number $B_3^{(2)}$.

\section{Connecting higher order Bell numbers and hyper partitions.}

The $m$-th order Bell numbers $B_n^{(m)}\quad (n=0,1,\ldots)$, studied by E. T. Bell in \cite{B38}, are given by the exponential generating functions

$$E_m(x)=\sum\limits_{n=0}^\infty B_n^{(m)}\frac{x^n}{n!}\quad (m\geq 1),$$

\noindent where $E_1(x)=\exp(\exp(x)-1)$ and $E_{m+1}(x)=\exp(E_m(x)-1)$.\\

\noindent In Table \ref{tab:bell} $B_n^{(m)}$ is computed for a few values of $m$ and $n$.\\

\begin{table}[htbp]
\begin{center}
\begin{tabular}{c||c|c|c|c|c|c|c|c}
$m\backslash n$& 1&2 &3&4&5&6&7&8\\
\hline
\hline
1&1& 2& 5& 15& 52& 203& 877& 4140\\
2&1& 3& 12& 60& 358& 2471& 19302& 167894\\
3&1& 4& 22& 154& 1304& 12915& 146115& 1855570\\
4&1& 5& 35& 315& 3455& 44590& 660665& 11035095\\
5&1& 6& 51& 561& 7556& 120196& 2201856& 45592666
\end{tabular}
\end{center}
\caption{Higher order Bell numbers $B_n^{(m)}$ when $1\leq n\leq 8$ and $1\leq m\leq 5$}
\label{tab:bell}
\end{table}

\begin{theorem}
The number of $m$-th order partitions of an $n$-set is $B_n^{(m)}$ ($m,n\geq 1$), i.e.
\begin{equation}
|\wp_n^{(m)}|=B_n^{(m)}.	
\label{eq:mainresult}
\end{equation}
\label{theorem:mainresult}
\end{theorem}

\begin{proof}
The proof makes use of generating functions. Recall that there are altogether $S(n,k)$ (Stirling number of the second kind) distinct $k$-partitions of the given $n$-set, i.e. there are $S(n,k)$ elements in $\wp_n^{(1)}$ which are $k$-sets. Each of these $k$-sets gives rise to $|\wp_k^{(m)}|$ distinct partitions of order $m+1$ of the $n$-set we started with, i.e. elements in $\wp_n^{(m+1)}$. Furthermore, different elements in $\wp_n^{(1)}$ of course give different elements in $\wp_n^{(m+1)}$, since they are already different at the "ground level". This means that we have the recurrence formula

\begin{equation}
|\wp_n^{(m+1)}|=\sum_{k=1}^n|\wp_k^{(m)}|S(n,k),\quad |\wp_k^{(0)}|=1.
	\label{eq:hyperpartition}
\end{equation}

\noindent Now, let $P^{(m)}(x)=\sum_{n=1}^\infty|\wp_n^{(m)}|\cdot x^n/n!$ denote the exponential generating function of $\{|\wp_{n}^{(m)}|\}_{n=1}^\infty$. Multiplication with $x^n/n!$, summation over $n$ and changing the order of summation in \eqref{eq:hyperpartition}, leads to

\begin{equation}
P^{(m+1)}(x)=\sum_{k=1}^\infty|\wp_k^{(m)}|\sum\limits_{n=1}^\infty S(n,k)\cdot \frac{x^n}{n!}.
	\label{eq:genfunpart1}
\end{equation}

\noindent It is well known that $\sum_{n=1}^\infty S(n,k)\cdot x^n/n!=(e^x-1)^k/k!$ is the exponential generating function of the Stirling numbers of the second kind. (A proof of this fact is included in Example \ref{ex:egf_stirling}.) From \eqref{eq:genfunpart1} we therefore get the recurrence formula

\begin{equation}
P^{(m+1)}(x)=\sum_{k=1}^\infty|\wp_k^{(m)}|\cdot \frac{(e^x-1)^k}{k!}=P^{(m)}(e^x-1).
	\label{eq:genfunpart2} 
\end{equation}

\noindent Now, since $P^{(0)}(x)=e^x-1$, a straightforward induction argument yields $P^{(m)}(x)=E_m(x)-1$, which proves \eqref{eq:mainresult} as well as the relation 
\begin{equation}
B_n^{(m+1)}=\sum_{k=1}^nB_k^{(m)}S(n,k),\quad B_k^{(0)}=1, 
	\label{eq:bellrelation1}
\end{equation}

\noindent which now follows from \eqref{eq:hyperpartition}.
\end{proof}

\begin{example}\label{ex:bell1}
\normalfont To demonstrate the power of hyper partition thinking, we give a combinatorial proof of the relation
\begin{equation}
	B_n^{(m)}=\sum\limits_{s=0}^{n-1}\binom{n-1}{s}B_s^{(m)}B_{n-s}^{(m-1)},\quad B_0^{(m)}=1,
	\label{eq:bellrelation2}
\end{equation}
\noindent that appeared in \cite[p. 545, (2.11)]{B38}.\\

\noindent Each element (i.e. a hyper partition) in $\wp_n^{(m)}$ consists of nested sets where the sets at the ground level are subsets of the basic $n$-set $S$. For a partition $p\in\wp_n^{(m)}$ we call elements of $S$ \textit{related} if they "reside" in the same outer set (box) in $p$. Now, fix an arbitrary $a\in S$. We count the number of partitions in $\wp_n^{(m)}$ according to which elements $a$ is related: Let $A$ be a basic $(n-s)$-set, including $a$, of related elements. Now for each of these $\binom{n-1}{n-s-1}=\binom{n-1}{s}$ sets there are $B_{n-s}^{(m-1)}$ inner structures. For the complementary $s$-set $C=S\setminus A$, the partitioning process will generate $B_s^{(m)}$ partitions of order $m$. By combining the possibilities, \eqref{eq:bellrelation2} follows.\\

Observe that if we in the same manner as above fix two (or more) elements in $S$, new formulas emerge.\\

By putting $B_n^{(0)}=1$, we note that \eqref{eq:bellrelation2} is a generalized version of the well known formula (see \cite[p. 210]{C74})

$$B_n=\sum\limits_{s=0}^{n-1}\binom{n-1}{s}B_s.$$ 
 
\end{example}

\section{Higher order Stirling numbers.}

Having established the relationship between partitions of order $m$ and higher order Bell numbers, defining higher order Stirling numbers seems like a natural thing to do.

\begin{definition}
The $m$-th order Stirling number $S^{(m)}(n,k)$ (of the second kind) is the number of $k$-sets in $\wp_n^{(m)}$.
\label{def:gen_stirling}
\end{definition} 

In \cite{B39} E. T. Bell gave an analytical definition of what he called \textit{generalized Stirling numbers} $\zeta_n^{(k,m)}$ by means of generating functions. In Theorem \ref{theorem:gen_stirling} we prove that $S^{(m)}(n,k)=\zeta_n^{(k,m)}$.\\

We note that the higher order Stirling numbers $S^{(m)}(n,k)$ are the entries of the matrix $\textbf{S}^m$, where $\textbf{S}=(S(n,k))$. This can be seen by induction from the relation \eqref{eq:higher_stirling} in the proof of Theorem \ref{theorem:gen_stirling}. Table \ref{tab:higher_stirling} is computed with the aid of such matrices.

\begin{table}[htbp]
\begin{center}
\begin{tabular}{c|c|c|c|c|c|c}
$m$& $S^{(m)}(5,1)$&$S^{(m)}(5,2)$&$S^{(m)}(5,3)$&$S^{(m)}(5,4)$&$S^{(m)}(5,5)$&$B_5^{(m)}$\\
\hline
\hline
5&3455&3325&725&50&1&7556\\
20&1115320&233050&11900&200&1&1360471\\
50&45533300&3706375&74750&500&1&49314926
\end{tabular}
\end{center}
\caption{Some examples of higher order Stirling and Bell numbers}
\label{tab:higher_stirling}
\end{table}

\begin{theorem}
Let $S^{(m)}(n,k)$ be the $m$-th order Stirling numbers of the second kind. Then we have 
$$S^{(m)}(n,k)=\zeta_n^{(k,m)},$$
where $\zeta_n^{(k,m)}$ are the generalized Stirling numbers of the second kind, defined by E.T. Bell in \cite[p. 91]{B39} by the generating functions
$$\frac{(E_{m-1}(t)-1)^k}{k!}=\sum\limits_n \zeta_n^{(k,m)}\cdot\frac{t^n}{n!}.$$
\label{theorem:gen_stirling}
\end{theorem}
	
\begin{proof}
Let $F_k^{(m)}$ denote the generating function of $\{S^{(m)}(n,k)\}_{n=k}^\infty$. Then we have
$$F_k^{(m)}(x)\equiv \sum\limits_{n=k}^\infty S^{(m)}(n,k)\frac{x^n}{n!}.$$ 
The idea is to come up with an analogous formula to \eqref{eq:hyperpartition}, in order to obtain an analogous formula to \eqref{eq:genfunpart2}. We claim that 
	\begin{equation}
	S^{(m+1)}(n,k)=\sum\limits_{i=k}^n S^{(m)}(i,k)S(n,i).
		\label{eq:higher_stirling}
	\end{equation}
	
\noindent This is true because we know that each of the $S(n,i)$ first order partitions will generate $S^{(m)}(i,k)$ $k$-partitions of order $m+1$. Multiplication by $x^n/n!$ followed by summation over $n$ in \eqref{eq:higher_stirling} gives
	$$F_k^{(m+1)}(x)=\sum\limits_{n,i} S^{(m)}(i,k)S(n,i)\frac{x^n}{n!}=\sum\limits_i S^{(m)}(i,k)\frac{(e^x-1)^i}{i!}=F_k^{(m)}(e^x-1),$$
because $\sum_n S(n,i)x^n/n!=(e^x-1)^i/i!$. We have thus deduced the recurrence formula 
	$$F_k^{(m+1)}(x)=F_k^{(m)}(e^x-1).$$ 
	
\noindent	Since $F_k^{(1)}(x)=(e^x-1)^k/k!$, induction yields 
	
$$F_k^{(m)}(x)=\frac{(E_{m-1}(x)-1)^k}{k!},\quad (E_0(x)=e^x),$$

\noindent which completes the proof.
\end{proof}

\noindent We notice that Theorem \ref{theorem:mainresult} is proved once more since we have $|\wp_n^{(m)}|=\sum_{k=1}^nS^{(m)}(n,k)$ and

$$P^{(m)}(x)=\sum\limits_{n=1}^\infty |\wp_n^{(m)}|\frac{x^n}{n!}=\sum\limits_{k=1}^\infty F_k^{(m)}(x)=E_m(x)-1.$$ 

\noindent How the introduction of the higher order Stirling numbers opens up the scope for hyper partition thinking, is illustrated in the next examples.

\begin{example}
\normalfont If we, in our construction process of $\wp_n^{(m)}$, stop at the $r$-th intermediate stage, i.e. in $\wp_n^{(r)}$, making up status so far by grouping the elements according to their cardinality before advancing further on, we get the following generalized version of \eqref{eq:higher_stirling}:
$$S^{(m)}(n,k)=\sum\limits_{i=k}^nS^{(m-r)}(i,k)S^{(r)}(n,i),$$
since there are $S^{(r)}(n,i)$, $i$-sets in $\wp_n^{(r)}$. This also follows from the matrix representation $\textbf{S}^m=\textbf{S}^{m-r}\textbf{S}^r$, as well as from \eqref{eq:higher_stirling} by induction.\\

\noindent Summing from $k=1$ to $n$ yields
	$$B_n^{(m)}=\sum\limits_{i=1}^nB_i^{(m-r)}S^{(r)}(n,i),$$
which generalizes \eqref{eq:bellrelation1}.
\end{example}

\begin{example}
\normalfont The formula 
\begin{equation}
S^{(m)}(n,k)=\sum\limits_{s=k-1}^{n-1}\binom{n-1}{s}B_{n-s}^{(m-1)}S^{(m)}(s,k-1)
\label{eq:stirlingrelation}	
\end{equation}
may be proved combinatorially in just the same manner as \eqref{eq:bellrelation2} in Example \ref{ex:bell1}. Note that \eqref{eq:stirlingrelation} yields \eqref{eq:bellrelation2} by summation over $k$.
\end{example}

\begin{example}
\normalfont Counting the $k$-sets in $\wp_n^{(m)}$ by first forming the $k$ "families" (outer sets) of related elements, we get
\begin{equation}
S^{(m)}(n,k)=\frac{1}{k!}\sum\limits_{\scriptsize i_1+\cdots+i_k=n}\binom{n}{i_1,\ldots,i_k}B_{i_1}^{(m-1)}\cdots B_{i_k}^{(m-1)},	
\label{eq:outside_in_1}
\end{equation}
where in this case $B_k^{(0)}=1,\; k\geq 1$ and $B_0^{(m)}=0$, because a family with $i$ relatives yields $B_i^{(m-1)}$ elements in $\wp_i^{(m-1)}$, i.e. there are exactly $B_i^{(m-1)}$ possible "inner" structures for an "$i$-family".\\
Summing over $k$ yields 
\begin{equation}
B_n^{(m)}=\sum\limits_{k=1}^n\frac{1}{k!}\sum\limits_{\scriptsize i_1+\cdots+i_k=n}\binom{n}{i_1,\ldots,i_k}B_{i_1}^{(m-1)}\cdots B_{i_k}^{(m-1)}.
\label{eq:outside_in_2}	
\end{equation}

\noindent Note that we have used nothing but the set-theoretic hyper partition interpretation/definition of $B_n^{(m)}$ and $S^{(m)}(n,k)$ in establishing \eqref{eq:outside_in_1} and \eqref{eq:outside_in_2}. Now, therefore, let $f_{m-1}(t)=f_{m-1}$ be the exponential generating function (e.g.f.) to $\{B_n^{(m-1)}\}_{n=0}^\infty$. And observe then that $f_{m-1}^k$ is the e.g.f. of the sequence

$$\left\{\sum\limits_{\scriptsize i_1+\cdots+i_k=n}\binom{n}{i_1,\ldots,i_k}B_{i_1}^{(m-1)}\cdots B_{i_k}^{(m-1)}\right\}_{n=0}^\infty.$$

\noindent Now we multiply \eqref{eq:outside_in_1} and \eqref{eq:outside_in_2} with $x^n/n!$, sum over $n$ and change the order of summation, to obtain
$$F_k^{(m)}(x)=\frac{1}{k!}f_{m-1}^k(x)\quad\mbox{and}\quad f_m(x)=\sum\limits_{k=1}^\infty\frac{f_{m-1}^k(x)}{k!}=\exp(f_{m-1}(x))-1.$$
And since $f_0(x)=e^x-1$, we have $f_m(x)=E_m(x)-1$. So by this reasoning "from outside in", we have an alternative proof of Theorem \ref{theorem:mainresult} as well as of Theorem \ref{theorem:gen_stirling}.
\label{ex:egf_stirling}
\end{example}

\section{An asymptotic consideration}

In \cite[p. 545]{B38} E.T. Bell proved the interesting formula

\begin{equation}
B_n^{(m)}=c_{n-1}m^{n-1}+c_{n-2}m^{n-2}+\cdots+c_0,	
\label{eq:bell_polynomial}
\end{equation}

\noindent where $c_{n-1},\ldots,c_0$ are rational numbers, independent of $m$. When $n$ is fixed, this implies that 

\begin{equation}
\lim\limits_{m\rightarrow\infty}\frac{B_n^{(m)}}{B_n^{(m-1)}}=1,	
\label{eq:bell_m_asymptotic}
\end{equation}

\noindent enabling us to say something more about the cardinality of the members of $\wp_n^{(m)}$.

\begin{remark}
Via a slightly different proof of \eqref{eq:bell_polynomial} than in \cite{B38} one might show that $c_{n-1}=n!/2^{n-1}$, see \cite{S19}.
\end{remark}

\noindent When looking at the "children" of a $p\in \wp_n^{(m)}$, i.e. the elements in $\wp_n^{(m+1)}$ that $p$ gives rise to, we see that all but one of them have lower cardinality than their "father" $p$. Following the next generations in the partitioning process, it appears that the great majority of the descendants are 1-element sets (see Table \ref{tab:higher_stirling}). It is therefore easy to conjecture that the average value $A_n^{(m)}$ of the cardinality of the sets in $\wp_n^{(m)}$ approaches 1 as $m\rightarrow\infty$, i.e. 

\begin{equation}
A_n^{(m)} =\frac{1}{B_n^{(m)}}\sum\limits_{k=1}^nkS^{(m)}(n,k)\underset{m\rightarrow\infty}{\longrightarrow} 1.	
\label{eq:average_cardinality}
\end{equation}

\noindent Let us see why \eqref{eq:average_cardinality} is true. \eqref{eq:bell_m_asymptotic} yields, in conjunction with the observation $S^{(m)}(n,1)=B_n^{(m-1)}$, that $S^{(m)}(n,1)\sim B_n^{(m)}\quad (m\rightarrow\infty)$. Since $B_n^{(m)}=\sum_{k=1}^n S^{(m)}(n,k)$, we consequently have

	$$S^{(m)}(n,k)=o(B_n^{(m)}),\quad k\geq 2\quad (m\rightarrow\infty),$$
	
\noindent and \eqref{eq:average_cardinality} follows immediately.

\section{A summary comment}
The set-theoretic intepretation of the higher order Bell numbers places them in a natural and fundamental context. Together with the higher order Stirling numbers, they take the same central position in the combinatorics of the higher order partition sets as their predecessors $B_n$ and $S(n,k)$ have at the ground level. 




\begin{thebibliography}{9}
\footnotesize
\bibitem{B38}
Bell, E.T., The Iterated Exponential Integers, \textit{Annals of mathematics,}{ \bf Vol. 39}, July 1938, 539 - 557.
\bibitem{B39}
Bell, E.T., Generalized Stirling transforms of sequences, \textit{Amer. Journal of Math,}{ \bf Vol. 61}, 1939, 89 - 101.
\bibitem{C74}
Comtet, L., Advanced Combinatorics, \textit{Reidel}, 1974.
\bibitem{S19}
Skau, I.H., Kristensen, K.F., An asymptotic Formula for the iterated exponential Bell Numbers, \url{http://arxiv.org/abs/1903.07979}, 2019.

\end{thebibliography}
\end{document}